\documentclass[twoside,11pt]{amsart} 
\usepackage{enumitem}
\usepackage{epsfig}
\usepackage{latexsym} 
\usepackage{amsmath} \usepackage{amssymb}
\usepackage{amsfonts,amssymb,amsmath,amscd,amsthm}
\usepackage{epsfig}
\usepackage[dvipsnames]{xcolor}
\usepackage{booktabs}
\usepackage{placeins}   % for FloatBarrier command



\makeatother

\newtheorem{lemma}{Lemma}

\theoremstyle{definition}

\newcommand{\pgap}   {{\mathcal G}}
\newcommand {\gap}     {\makebox[0.075 in]{}}

\newcommand {\fto}     {\longrightarrow}

\newcommand {\set}[1]  {\left\{ {#1} \right\}}

\begin{document}

\title[On nonconvex constellations among primes II]{On nonconvex constellations \\ among primes II: $(J,|s|)=(458,3240)$ }

\date{19 May 2026}

\author{Fred B. Holt}
% {\smallit Department of Mathematics, One University, Somewhere, State, Country}\\
\address{\tt fbholt@primegaps.info, www.github.com/fbholt/Primegaps-v2}
																														
\begin{abstract}
Extending our work on the $k$-tuple conjecture, we previously applied those methods to the Engelsma counterexamples (narrow constellations)
of length $J=459$ and span $|s|=3242$.  Here we extend that analysis to the $116$ Engelsma counterexamples of length $J=458$
and $|s|=3240$.

We track the evolution of these $116$ counterexamples from inadmissible driving terms starting in the cycle of gaps $\pgap(11^\#)$
up through their first appearance in $\pgap(113^\#)$.  
We continue developing primorial coordinates for each admissible instance
through a breadth-first exhaustive search through $\pgap(211^\#)$.

Each of the $(458,3240)$ constellations sits inside a $(459,3242)$ constellation, which we call its {\em parent}.  We show that
no $(458,3240)$ constellation occurs outside of its parent until the cycle $\pgap(227^\#)$.
The early evolution of the $(458,3240)$ constellations is dominated by the evolution of their parents, which we have previously studied.

For each $(458,3240)$-counterexample we calculate its asymptotic relative population, among other constellations
of length $J=458$.
\end{abstract}

\maketitle
\pagestyle{myheadings}
\thispagestyle{empty}
\baselineskip=12.875pt
\vskip 30pt

\section{Introduction}
This paper continues our study of admissible nonconvex constellations \cite{FBHnonconvex}.  In that first paper we established the 
basics for this study, including establishing that the smallest nonconvex constellations discovered by Engelsma \cite{Engtbl, Sutherland} are
$$ (J,|s|) \; = \; (458,3240), \; {\rm and} \gap (459,3242).$$
We previously \cite{FBHnonconvex} studied the $58$ nonconvex constellations in $(459,3242)$.  Here we leverage that work to study the $116$ nonconvex
constellations in $(458,3240)$.

The prior work \cite{FBHnonconvex} provides the background for the present study.

An admissible constellation $s$ of length $J$ is said to be {\em nonconvex} iff 
$$ \pi(|s|) < J.$$
That is, the number of small primes in the interval $(0,|s|]$ is less than the number of gaps in this constellation.  Figure~\ref{ConvexFig}
illustrates this.

\begin{figure}[hbt]
\centering
\includegraphics[width=4in]{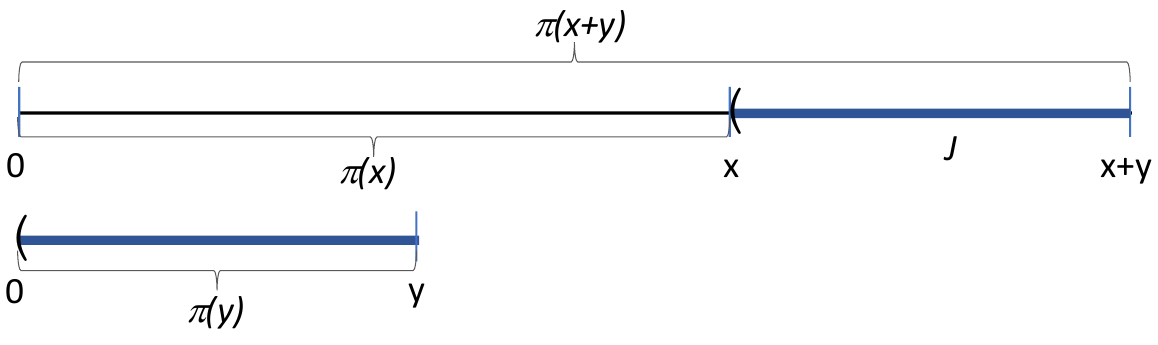}
\caption{\label{ConvexFig} Convexity and nonconvexity.  We are comparing the density of primes in an interval away from the origin, with the
density of an interval of equal length at the origin. }
\end{figure}

We have previously shown \cite{FBHktuple} that all admissible constellations
arise and persist in the cycles of gaps $\pgap(p^\#)$ among the $p$-rough numbers, 
and that the population of {\em every} admissible constellation of length $J$ ultimately grows
as ${\Theta(\prod_{p > J+1}(p-J-1))}$ within these cycles.  So the Engelsma constellations in $(458,3240)$  and $(459,3242)$ do arise within
the cycles $\pgap(p^\#)$.  In this paper we track their early co-evolution.

No actual instance of a nonconvex constellation among primes has yet been discovered.
That is, nonconvex admissible constellations $s$ have been identified, but we do not yet know of a specific instance $\gamma_0$
where $s$ occurs among primes.

We have the narrow constellation $s$, but we still need the initial generator $\gamma_0$ for an instance of $s$ among primes.
If any of the nonconvex constellations occur among primes, these instances would be counterexamples to Hardy and
Littlewood's convexity conjecture from 1923 (\cite{HL}, p.54).

We recently \cite{FBHnonconvex} looked at the $58$ $(459,3242)$ constellations identified by Engelsma.  These constellations all begin and end
with the gap $g=2$.  If we drop the ending gap $g=2$ from these, we get $58$ of the $(458,3240)$ constellations; and if we
drop the starting gap $g=2$ from the $(459,3242)$ constellations we get the other $58$ of the $(458,3240)$ constellations.
So every instance of a $(459,3242)$ constellation contains a pair of $(458,3240)$ constellations.

For a $(J, |s|)$-counterexample $s$, we call its minimal extension to length $J+1$ its {\em parent}.  The $(458,3240)$ constellations all have
parents that are $(459,3242)$-counterexamples.

Under the three-step recursion \cite{FBHSFU, FBHktuple} across the cycles of gaps $\pgap(p^\#)$ in Eratosthenes sieve, we know that
there are no admissible driving terms for any of the $(458,3240)$ or  $(459,3242)$ constellations, other than the constellations themselves.

Let $s$ be one of the $(459,3242)$ constellations.  If an instance of $s$ is fused on the left side of the initial gap $g=2$ or on 
the right side of the terminal gap $g=2$, we thereby obtain an instance of a $(458,3240)$ constellation outside of a $(459,3242)$
constellation.  So as Eratosthenes sieve continues, instances of the $(458,3240)$ constellations evolve both inside and outside
of the $(459,3242)$ constellations. 

\begin{figure}[hbt]
\centering
\includegraphics[width=5in]{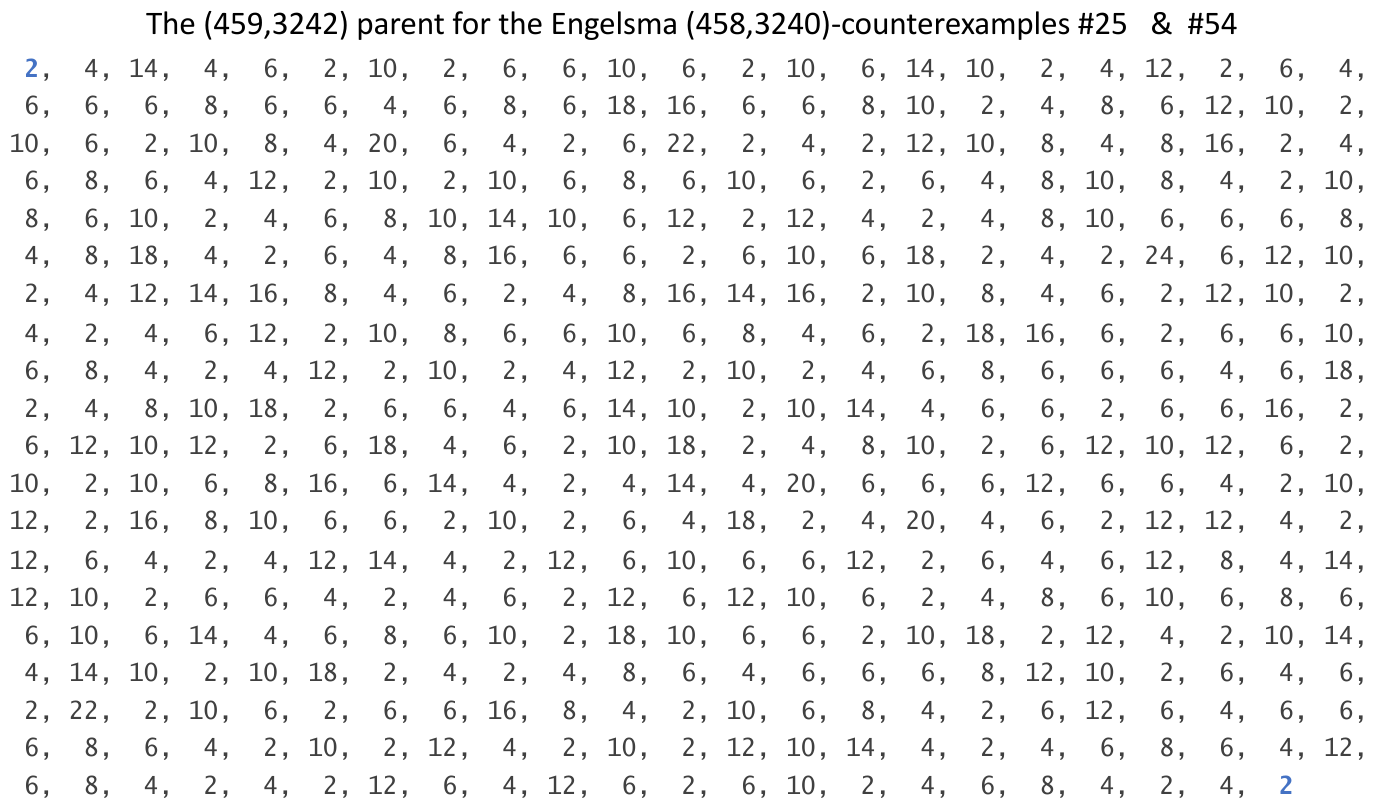}
\caption{\label{Eng458Fig} An example of two $(458,3240)$-constellations  $[2 \hat{s}]_{25}$ and  $[\hat{s} 2]_{54}$ sitting inside their
$(459,3242)$ parent.  $[2 \hat{s}]_{25}$ drops the last gap of $2$ from its parent, and $[\hat{s} 2]_{54}$ drops
the initial gap $2$.}
\end{figure}

In this paper we provide a brief analysis of the $116$ $(458,3240)$ constellations, consistent with our analysis of the $58$
$(459,3242)$ constellations.  Since $s_{25}$ has the highest asymptotic relative population among the $(458,3240)$ constellations, we
use it for our running example.

The raw data for the Engelsma counterexamples is available at Sutherland's website \cite{Eng2}.   The computations introduced here, including the
sorting of the constellations by the prefixes of their primorial coordinates, is available on GitHub.\footnote{https://github.com/fbholt/primegaps-v2}

% SECTION: 458,3240 and their parents  ====================================
\section{Engelsma $(458,3240)$-counterexamples and their parents}
Engelsma identified $58$ $(459,3242)$-counterexamples, each beginning and ending with a gap $g=2$.  Of the $116$ $(458,3240)$-counterexamples 
half of these drop the last gap $g=2$ from a $(459,3242)$-counterexample and half of these drop the initial gap $g=2$.  
Figure~\ref{Eng458Fig} shows $[2 \hat{s}]_{25}$ and $[\hat{s} 2]_{54}$ sitting inside their parent $[2 \hat{s} 2]_{25}$. 
This is a natural 2-to-1 association between the $(458,3240)$-counterexamples and the $(459,3242)$-counterexamples.

We use this 2-to-1 correspondence often enough that we need a dedicated notation.  By inspecting the data \cite{Eng2}, we know that
every $(459,3242)$ constellation has the form $[2 \hat{s} 2]$ where $\hat{s}$ is a $(457,3238)$ constellation. 
For each $[2 \hat{s} 2]$ in $(459,3242)$, both $[2 \hat{s}]$ and $[\hat{s} 2]$ are $(458,3240)$ constellations.

The $58$ $(459,3242)$-counterexamples come in $29$ symmetric pairs under reversal of the constellations.  We index these constellations
by the unique prefixes of their primorial coordinates.  For the $29$ $(459,3242)$-counterexamples that begin $s=2 \, 4 \, 14 \, \ldots$,
in $\pgap(11^\#)$ there is a unique inadmissible 
driving term starting at $\gamma_0=107$.  

For the $29$ reversed $(459,3242)$-counterexamples, there is a unique inadmissible constellation
in $\pgap(11^\#)$ starting at $\gamma_0=1271$.  Under this ordering, the constellation of index $57-j$ is the reverse of the constellation with 
index $j$.

When we move toward the $(458,3240)$-counterexamples, we also index them by their prefixes in primorial coordinates.
These prefixes are listed in Figure~\ref{EngPrefixFig}.  Since all of the instances of the $(458,3240)$ constellations occur inside their
$(459,3242)$ until at least $\pgap(227^\#)$ and the prefixes of the unique instances only last into $\pgap(131^\#)$ or $\pgap(137^\#)$,
these prefixes are identical to the prefixes for their $(459,3242)$ parents, with two exceptions.  For those parents $[2 \hat{s} 2]$ that have
initial generator $\gamma_0=107$ in $\pgap(11^\#)$, their tails $[\hat{s} 2]$ have $\gamma_0=109$.  Similarly, for the parents that have
$\gamma_0=1271$, their tails $[\hat{s} 2]$ have $\gamma_0=1273$.

Indexing the constellations under the order of the unique prefixes for their primorial expansions, we have relations among the indices as illustrated in 
Figure~\ref{NonconmapFig}.

\begin{figure}[h]
\centering
\includegraphics[width=3.75in]{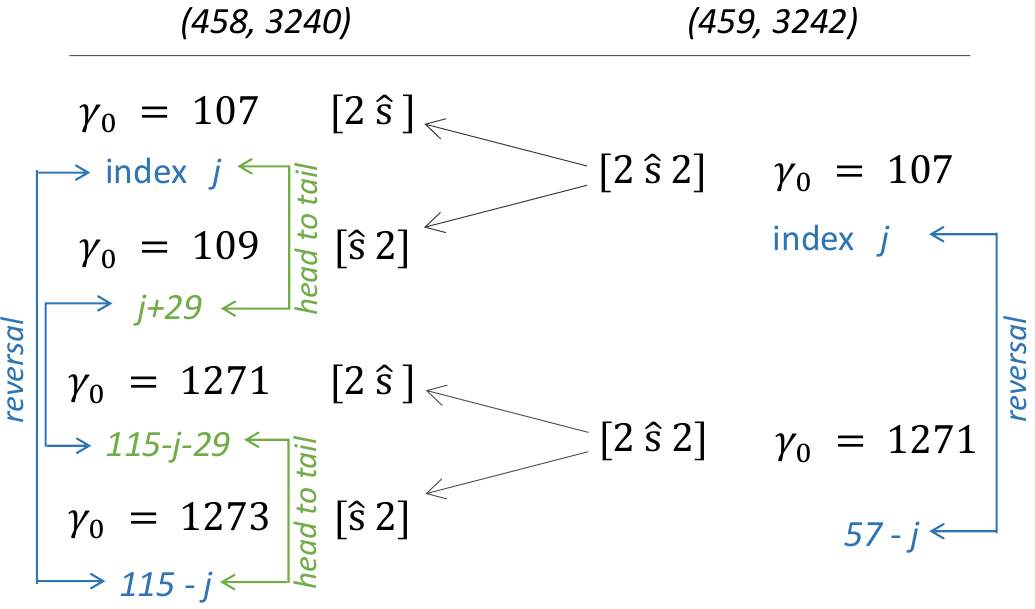}
\caption{\label{NonconmapFig} The relationships among the $(458,3240)$ and $(459,3242)$ constellations, when ordered by the prefixes for
their primorial coordinates.}
\end{figure}

We lean heavily on the analysis of the $(459,3242)$-counterexamples \cite{FBHnonconvex}, 
because no instance of a $(458,3240)$ counterexample occurs outside of its $(459,3242)$ parent until $\pgap(227^\#)$.

\begin{lemma}
Let $[2 \hat{s} 2]$ be a $(459,3242)$ counterexample of index ${0\le j \le 28}$, when the constellations are ordered by the prefixes of their primorial coordinates.

Then $[2 \hat{s}]$ is a $(458,3240)$ counterexample of index $j$, and $[2 \hat{s}]$ first occurs outside of its parent in $\pgap(227^\#)$. 

And $[\hat{s} 2]$ is  also a $(458,3240)$ counterexample of index $j+29$, and $[\hat{s} 2]$ first occurs outside of its parent in $\pgap(269^\#)$.
\end{lemma}

\begin{proof}
Regarding the indices, we note that for $0 \le j \le 28$ the constellation $[2 \hat{s} 2]$ has the value $\gamma_0=107$ in $\pgap(11^\#)$.
Likewise the constellation $[2 \hat{s}]$ has $\gamma_0=107$, and by computation its unique prefix will be that of its parent. 
See Figures~\ref{NonconmapFig} and~\ref{EngPrefixFig}.  
So the ordering of these constellations
is the same as for their parents, and if the parent has index $j$, the $(458,3240)$ constellation $[2 \hat{s}]$ will have index $j$.

The constellation $[\hat{s} 2]$ has $\gamma_0=109$, and after $\gamma_0$ its unique prefix will be that of its parent.  
Under ordering by their prefixes these constellations $[\hat{s} 2]$ will occur as a block after the block of $[2 \hat{s}]$, which have $\gamma_0=107$.
So if the parent $[2 \hat{s} 2]$ has index $j$, the constellation $[\hat{s} 2]$ will have index $j+29$.

To determine when  the constellations $[2 \hat{s}]$ and $[\hat{s} 2]$ occur outside of a parent, we look at their numbers of admissible
residues $\bmod \: p$, which we denote $\rho_s = p - \nu_p(s)$.  

Starting with $p=11$, we look at ${\delta = \rho_{[2 \hat{s}]} - \rho_{[2 \hat{s} 2]}}$ 
or ${\delta=\rho_{[\hat{s} 2]} - \rho_{[2 \hat{s} 2]}}$, for the heads and tails of $[2 \hat{s} 2]$ respectively. 
 Every instance of a parent $[2 \hat{s} 2]$ contains both $[2 \hat{s}]$ and $[\hat{s} 2]$, so 
both differences are non-negative.  These differences always have values $\set{0,1}$.

As long as these differences
are $0$, every instance of $[2 \hat{s}]$ and $[\hat{s} 2]$ occurs within its parent.  From the computation summarized in 
Figure~\ref{EngPrefixFig} we see that the unique prefixes extend into $\pgap(131^\#)$.  That is, there is a unique image of each $(458,3240)$ constellation
 in $\pgap(131^\#)$, sitting inside its $(459,3242)$ parent.  

Figure~\ref{EngAdmisFig} continues this analysis to the subsequent primes.
Figure~\ref{EngAdmisFig} tabulates the numbers of admissible residues for each $(458,3240)$ constellation, for primes $137 \le p \le 457$.
Cells are highlighted in green when $\delta=1$.  We observe that
for all $[2 \hat{s}]$, this first occurs for $p=227$, and for all $[\hat{s} 2]$ this first occurs for $p=269$.
\end{proof}

\begin{figure}[hbt]
\centering
\includegraphics[width=5.25in]{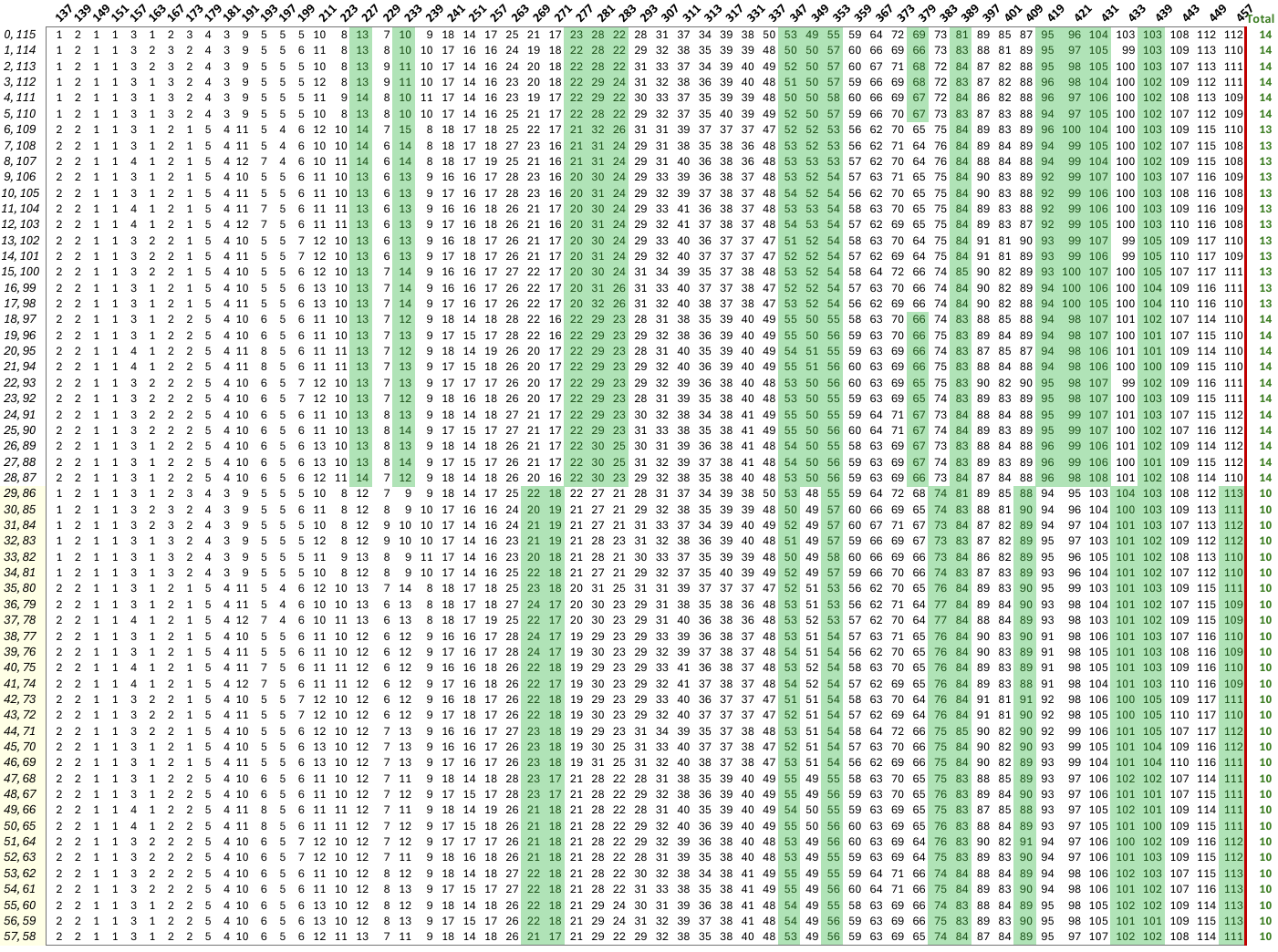}
\caption{\label{EngAdmisFig} These are the numbers $\rho_s$ of admissible residues $\bmod \: p$ for each of the $(458,3240)$-counterexamples,
 for primes ${137 \le p < 459}$.
The population of any one constellation in the cycle $\pgap(p^\#)$ is the product of these numbers up through $p$.  Green highlights the cycles in which
the number of admissible residues for $s$ exceeds the number for the parent $(459,3242)$-counterexample by $1$.  The leftmost of the green cells 
in each row indicate where instances of $(458,3240)$ first occur outside of its parent.}
\end{figure}

Table~\ref{Eng25nadm} provides a specific example.  We list the number ${\rho=p-\nu_p}$ of admissible residues $\bmod \: p$ for our
example $[2 \hat{s}]_{25}$ over the  primes $13 \le p \le 499$.  
We track the unique inadmissible driving term up through $q=109$.  The constellation $[2 \hat{s}]_{25}$ appears in $\pgap(113^\#)$
and has a unique instance ($\prod_{q\le p} \rho(q)=1$) into $\pgap(131^\#)$.  Then multiple admissible residues start appearing.
We highlight in green the primes $p$ for which there is one more admissible residue for $[2 \hat{s}]_{25}$ than for its parent
$[2 \hat{s} 2]_{25}$.  

The prime $p > 459$ marks a transition; for primes $p > 459$ there are more residue classes $\bmod \: p$ than 
the $459$ values possibly covered by any instance of a $(458,3240)$ constellation.  For the primes $p \le 459$ we need to demonstrate
admissibility by showing that $\rho >0$ for each prime.  For primes $p > 459$ the admissibility condition is trivially satisfied.

\begin{table}[htb]
\centering
\begin{tabular}{|rc||rc||rc||rc||rc||rc|}
\multicolumn{2}{c}{$p \hspace{0.2in} p-\nu_p$} & \multicolumn{2}{c}{$p \hspace{0.2in} p-\nu_p$} & \multicolumn{2}{c}{$p \hspace{0.2in} p-\nu_p$} & 
\multicolumn{2}{c}{$p \hspace{0.2in} p-\nu_p$} & \multicolumn{2}{c}{$p \hspace{0.2in} p-\nu_p$} & \multicolumn{2}{c}{$p \hspace{0.2in} p-\nu_p$} \\ \hline
{\em 13} & 1 & {\em 73} & 1 & 151 & 1 & \textcolor{ForestGreen}{\bf 233} & \textcolor{ForestGreen}{\bf 14} & 317 & 38 & \textcolor{ForestGreen}{\bf 419} & \textcolor{ForestGreen}{\bf 95} \\
{\em 17} & 1 & {\em 79} & 1 & 157 & 3 & 239 & 9 & 331 & 41 & \textcolor{ForestGreen}{\bf 421} & \textcolor{ForestGreen}{\bf 99} \\
{\em 19} &  1 & {\em 83} & 1 & 163 & 2 & 241 & 17 & 337 & 49 & \textcolor{ForestGreen}{\bf 431} & \textcolor{ForestGreen}{\bf 107} \\
{\em 23} &  1 & {\em 89} & 1 & 167 & 2 & 251 & 15 & \textcolor{ForestGreen}{\bf 347} & \textcolor{ForestGreen}{\bf 55} & 433 & 100 \\
{\em 29} &  1 & {\em 97} & 1 & 173 & 2 & 257 & 17 & \textcolor{ForestGreen}{\bf 349} & \textcolor{ForestGreen}{\bf 50} & \textcolor{ForestGreen}{\bf 439} & \textcolor{ForestGreen}{\bf 102} \\
{\em 31} & 1 & {\em 101} & 1 & 179 & 5 & 263 & 27 & \textcolor{ForestGreen}{\bf 353} & \textcolor{ForestGreen}{\bf 56} & 443 & 107 \\
{\em 37} & 1 & {\em 103} & 1 & 181 & 4 & 269 & 21 & 359 & 60 & 449 & 116 \\
{\em 41} & 1 & {\em 107} & 1 & 191 & 10 & 271 & 17 & 367 & 64 & 457  & 112  \\ \cline{11-12}
{\em 43} & 1 & {\em 109} & 1 & 193 & 6 & \textcolor{ForestGreen}{\bf 277} & \textcolor{ForestGreen}{\bf 22} & 373 & 71 & \textcolor{BlueGreen}{\bf 461} & \textcolor{BlueGreen}{\bf 115} \\ \cline{3-4}
{\em 47} & 1 & \textcolor{Maroon}{\bf 113} & \textcolor{Maroon}{1} & 197 & 5 & \textcolor{ForestGreen}{\bf 281} & \textcolor{ForestGreen}{\bf 29} & \textcolor{ForestGreen}{\bf 379} & \textcolor{ForestGreen}{\bf 67} & \textcolor{blue}{463} & \textcolor{blue}{128} \\
{\em 53} & 1 & 127 & 1 & 199 & 6 & \textcolor{ForestGreen}{\bf 283} & \textcolor{ForestGreen}{\bf 23} & 383 & 74 & \textcolor{blue}{467} & \textcolor{blue}{127} \\
{\em 59} & 1 & 131 & 1 & 211 & 11 & 293 & 31 & \textcolor{ForestGreen}{\bf 389} & \textcolor{ForestGreen}{\bf 84} & \textcolor{blue}{479} & \textcolor{blue}{135} \\
{\em 61} & 1 & 137 & 2 & 223 & 10 & 307 & 33 & 397 & 89 & \textcolor{BlueGreen}{\bf 487} & \textcolor{BlueGreen}{\bf 131} \\
{\em 67} & 1 & 139 & 2 & \textcolor{ForestGreen}{\bf 227} & \textcolor{ForestGreen}{\bf 13} & 311 & 38 & 401 & 83 & \textcolor{BlueGreen}{\bf 491} & \textcolor{BlueGreen}{\bf 148} \\
{\em 71} & 1 & 149 & 1 & 229 & 8 & 313 & 35 & 409 & 89 & \textcolor{BlueGreen}{\bf 499} & \textcolor{BlueGreen}{\bf 141} \\ \hline
\end{tabular}
 \caption{ \label{Eng25nadm}The number $\rho=p - \nu_p$ of admissible residues $\bmod p$ for the counterexample ${s_{25} \in (458,3240)}$ over the 
 primes ${13 \le p \le 499}$.  The evolution begins with inadmissible driving terms in the cycles up through $\pgap(109^\#)$.
The constellation $s_{25}$ itself first appears in $\pgap(113^\#)$, and there continues to be a unique image of $s_{25}$ until $\pgap(137^\#)$.}
\end{table}

% SECTION: co-evolution & prefixes  ====================================
\section{Co-evolution and primorial expansions}

The analysis of the $(458,3240)$ constellations is entwined with the analysis of their $(459,3242)$ parents.  We have seen above that they
are more entwined than one might expect.  No $[2 \hat{s}]$ with $\gamma_0 = 107$ occurs outside of its parent until $\pgap(227^\#)$.  
At this point their populations \cite{FBHktuple}
$$ n_s(p^\#) \; = \; \prod_{131 < q \le p} \rho_s(q) $$
in $\pgap(223^\#)$ range from $3.888 E7$ to $4.838 E8$ admissible instances, a formidable starting point for any exhaustive search or analysis.  

We start tracking the evolution in the cycle $\pgap(11^\#)$.  Of the $58$ $(459,3242)$-counterexamples, $29$ share a single
inadmissible driving term in $\pgap(11^\#)$ with $\gamma_0=107$.  The other $29$ are the mirror images of the first $29$, 
and these mirror images share a unique inadmissible driving term in $\pgap(11^\#)$ with $\gamma_0=1271$.

The $29$ constellations with $\gamma_0=107$ continue to share a single inadmissible instance into $\pgap(59^\#)$.
Then the evolutionary tree begins to fan out across the various $(459,3242)$ constellations, but each $(459,3242)$ constellation
and each $(458,3240)$ constellation continues to have a single inadmissible instance until $\pgap(113^\#)$, when all of these
admissible constellations actually occur.  

The admissible instances for each constellation remain unique across $\pgap(127^\#)$ and $\pgap(131^\#)$.  A few of them have unique admissible
instances in $\pgap(137^\#)$ as well.  
We record these
unique prefixes by their primorial coordinates in Figure~\ref{EngPrefixFig}.

Figure~\ref{EngPrefixFig} tabulates the primorial coordinates for the unique prefixes for the first $58$ $(458,3240)$-counterexamples.
The other $58$ $(458,3240)$-counterexamples are the mirror images of this first set.
As an example of how to use this table, here is the unique instance $\gamma_0(131^\#)$ for $[2 \hat{s}]_{25}$:
\begin{eqnarray*}
\gamma_0(131^\#) &= & 107 + 6\cdot 11^\# + 8\cdot 13^\# + 9\cdot 17^\# + 5\cdot 19^\# + 7\cdot 23^\# + 1\cdot 29^\#  \\
 & & \quad + 23\cdot 31^\#  + 38\cdot 37^\# + 34\cdot 41^\# + 46\cdot 43^\# + 20\cdot 47^\# + 13 \cdot 53^\# \\
 & & \quad + 13\cdot 59^\# + 39\cdot 61^\# + 42\cdot 67^\# + 45\cdot 71^\# + 54\cdot 73^\# + 82\cdot 79^\# \\
 & & \quad + 79\cdot 83^\# + 79\cdot 89^\# + 82\cdot 97^\# + 10\cdot 101^\# + 11\cdot 103^\#  \\
 & & \quad + 78\cdot 107^\# + 14\cdot 109^\#  +  74\cdot 113^\# + 55\cdot 127^\# \\
 & \approx & 2.2313949 \, E50
 \end{eqnarray*}
This prefix is in the highlighted row, index $j=25$, in Figure~\ref{EngPrefixFig}.

\begin{figure}[hbt]
\centering
\includegraphics[width=5.4in]{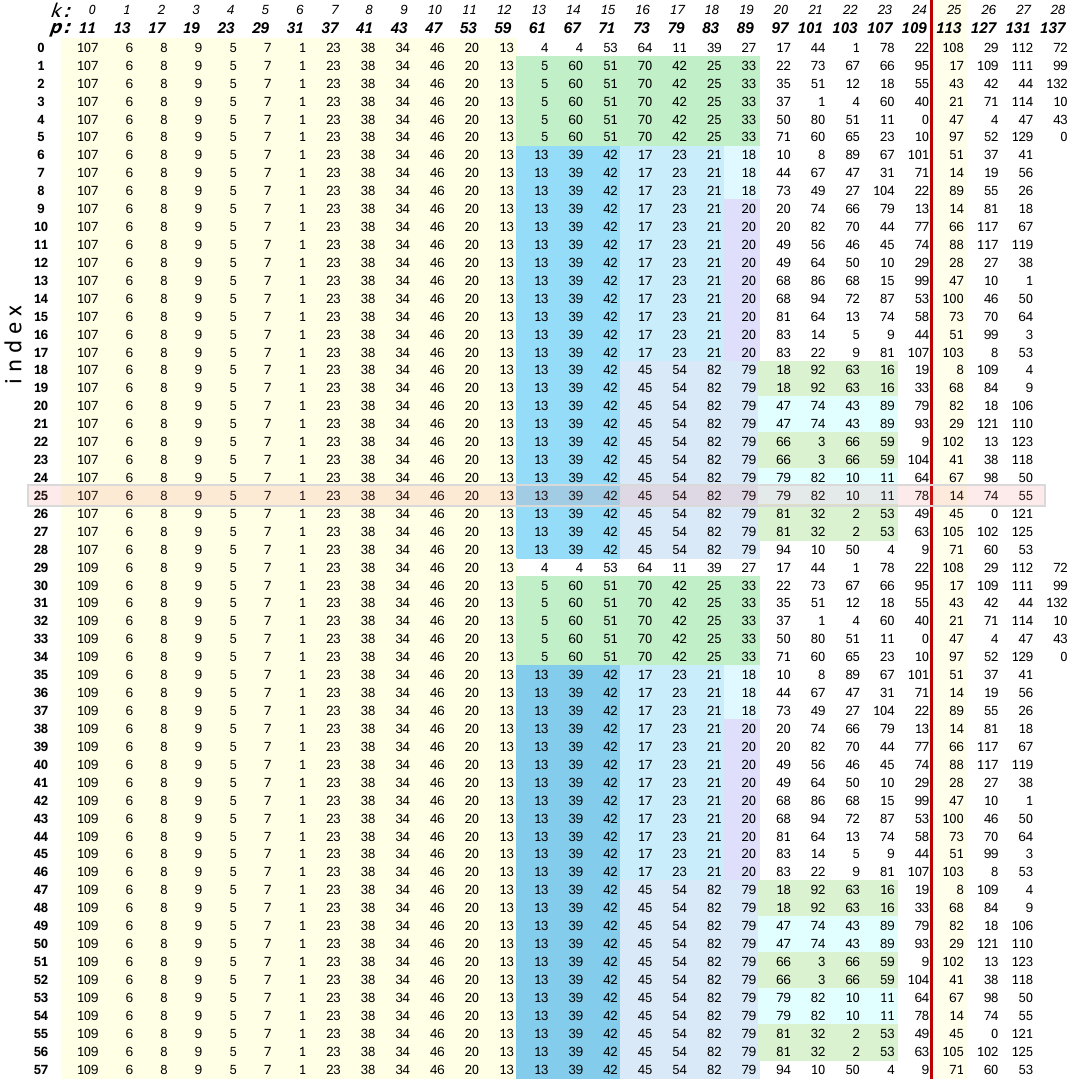}
\caption{\label{EngPrefixFig} Listed here are the primorial coordinates for the unique instances of each of the 58 
$(458,3240)$-counterexamples that have initial generator $\gamma_0=107$ or $\gamma_0=109$ in $\pgap(11^\#)$.  The driving terms
in $\pgap(p^\#)$ for $p < 113$ are inadmissible -- they do not survive.  The counterexamples themselves appear in $\pgap(113^\#)$.
The row of primorial coordinates for the prefix for our example $s_{25}$ is highlighted.}
\end{figure}

Beyond the unique prefixes listed in Figure~\ref{EngPrefixFig}, the instances of each counterexample branch out into a tree.
The number of nodes at first grows slowly but then explodes, as the product of the number of admissible instances ${\delta=p-\nu_p}$
at each stage of the sieve \cite{FBHktuple}.  Exhaustive breadth-first explorations of the instances of any of the
counterexamples run out very quickly.  

This is problematic when we try to consider instances of $(458,3240)$ constellations that occur outside of their $(459,3242)$ parents, as
noted above.  Specifically, if we are hoping to find an instance $\gamma_0$ of a $(458,3240)$ constellation that survives the sieve before the
parent survives \cite{FBHnonconvex}, we need access to this outside population.

% SUBSECTION  :  population model
\subsection{Population model for $(458,3240)$ constellations}
Since neither the $(458,3240)$ constellations or the $(459,3242)$ constellations have driving terms other than the constellations themselves,
the population dynamics \cite{FBHSFU} across stages of Eratosthenes sieve are pretty simple.
Here we segment the population of a $(458,3240)$ constellation $s$ in $\pgap(p^\#)$ into the population $n_{\rm out}(p^\#)$ of instances 
that occur outside of a parent and the population $n_{\rm in}(p^\#)$

We know \cite{FBHktuple} that eventually the ratio of the population of any $(458,3240)$ constellation to any $(459,3242)$ constellation 
in $\pgap(p^\#)$ will go to $0$.  We have also seen that all $(458,3240)$ constellations occur inside their parents through at least 
$\pgap(223^\#)$ which has span $3.67 E86$. 

We have the following model for the evolving populations of a $(458,3240)$ constellation inside and outside of its parent.

\begin{figure}[hbt]
\centering
\includegraphics[width=3.5in]{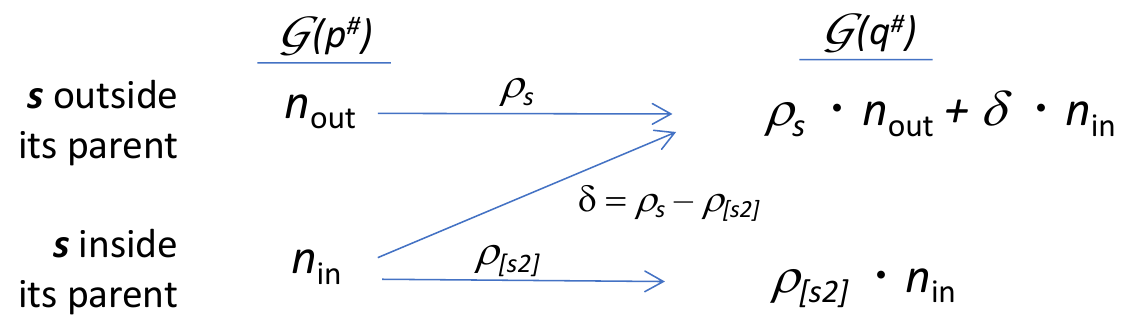}
\caption{\label{ParentSysFig} An illustration of the population dynamics for a $(458,3240)$ constellation and its $(459,3242)$ parent, across
a single stage of Eratosthenes sieve.  Here $p$ and $q$ are consecutive primes, and $\rho_s = q-\nu_q$ is the number of admissible residues
$\bmod q$ for $s$.}
\end{figure}

\begin{lemma}
Let $s$ be a $(458,3240)$ constellation, and denote its $(459,3242)$ parent by $[s 2]$.
Let $p$ and $q$ be consecutive primes, and let $\rho_s = q-\nu_q(s)$ be the number of admissible residues $\bmod q$ for $s$. 

Then the population $n_{\rm out}$ of $s$ outside of its $(459, 3242)$ parent and its population $n_{\rm in}$ inside its parent grow as
$$
\left[ \begin{array}{c} n_{\rm out}(q^\#) \\ n_{\rm in}(q^\#) \end{array} \right] \; = \;
\left[ \begin{array}{cc}
\rho_s & \delta \\
0 & \rho_{[s 2]} \end{array} \right]
\left[ \begin{array}{c} n_{\rm out}(p^\#) \\ n_{\rm in}(p^\#) \end{array} \right] .
$$
Here $\delta = \rho_s - \rho_{[s 2]} \in \set{0,1}$, with $\delta = 0$ for $q < 227$ and $\delta = 1$ for $q > 1621$.
\end{lemma}

\begin{proof}
Figure~\ref{ParentSysFig} illustrates the dynamics we describe in this proof.

Let $\gamma_0$ be an instance of $s$ in $\pgap(p^\#)$.  If the instance $\gamma_0$ is outside of its parent, then all $\rho_s$ images of
this instance in $\pgap(q^\#)$ remain outside of the parent constellation.

If the instance $\gamma_0$ is inside its parent in $\pgap(p^\#)$, then $\rho_{[s 2]}$ images of this instance occur in images of the parent
in $\pgap(q^\#)$.  If $\rho_{[s 2]} = \rho_s$, then this accounts for all the images of this instance $\gamma_0$ of $s$ in $\pgap(q^\#)$.
If instead $\rho_{[s 2]} < \rho_s$, then there is one image of this instance of $s$ that persists into $\pgap(q^\#)$ when the extremal gap $2$ is
removed from the parent.

From direct calculations of $\rho_s$ and $\rho_{[s 2]}$ across the 59 $(458,3240)$ constellations we know that $\delta=0$ for $q < 227$. 
See Figure~\ref{EngAdmisFig}.

On the other hand, under the recursion $\pgap(p^\#) \fto \pgap(q^\#)$ \cite{FBHSFU}, once $q > 1621$ all of the fusions between
adjacent gaps occur in separate images of $s$ and $[s 2]$.  So for $q > 1621$ we have $\rho_s= q-459$ and $\rho_{[s 2]}=q-460$, 
and thereafter $\delta = 1$.
\end{proof}

In Figure~\ref{ParentSysFig} we plot the proportion of occurrences of the $(458,3240)$ constellations that are inside their parents.
The first $58$ constellations are plotted; the other $58$ are their reversals.  Under symmetry of the cycle $\pgap(p^\#)$, the
$(458,3240)$ constellations of index $j$ and $115-j$ share the same curve within Figure~\ref{ParentSysFig}.  

We mark the thresholds $J+1 = 459$ and $|s|/2=1620$ in the graph.  The first threshold marks the point at which there are more residue
classes $\bmod \: p$ than there are possible fusions within the constellation $s$.  At this point bare admissibility becomes trivial.  The 
second threshold marks the point -- half the span of the constellation -- at which all fusions of adjacent gaps must occur in separate images
of $s$ under the 3-step recursion.  Beyond this point $\delta = 1$, and the curves all decline by factors of $\frac{p-460}{p-459}$.

Before that first threshold $J+1=459$, we see the effect of the data tabulated in Figure~\ref{EngAdmisFig}.  Over the interval $131 \le p \le 459$
the constellations $s_0, \ldots, s_{28}$ have $\delta=1$ either $13$ or $14$ times, and these first occur for $p=227$ and $p=233$.  In contrast,
over this interval the second set of constellations $s_{29}, \ldots, s_{57}$ each have $\delta=1$ only $10$ times, and this doesn't begin until
$p=269$ and $p=271$.  The individual graphs separate into two bundles accordingly.

\begin{figure}[hbt]
\centering
\includegraphics[width=5in]{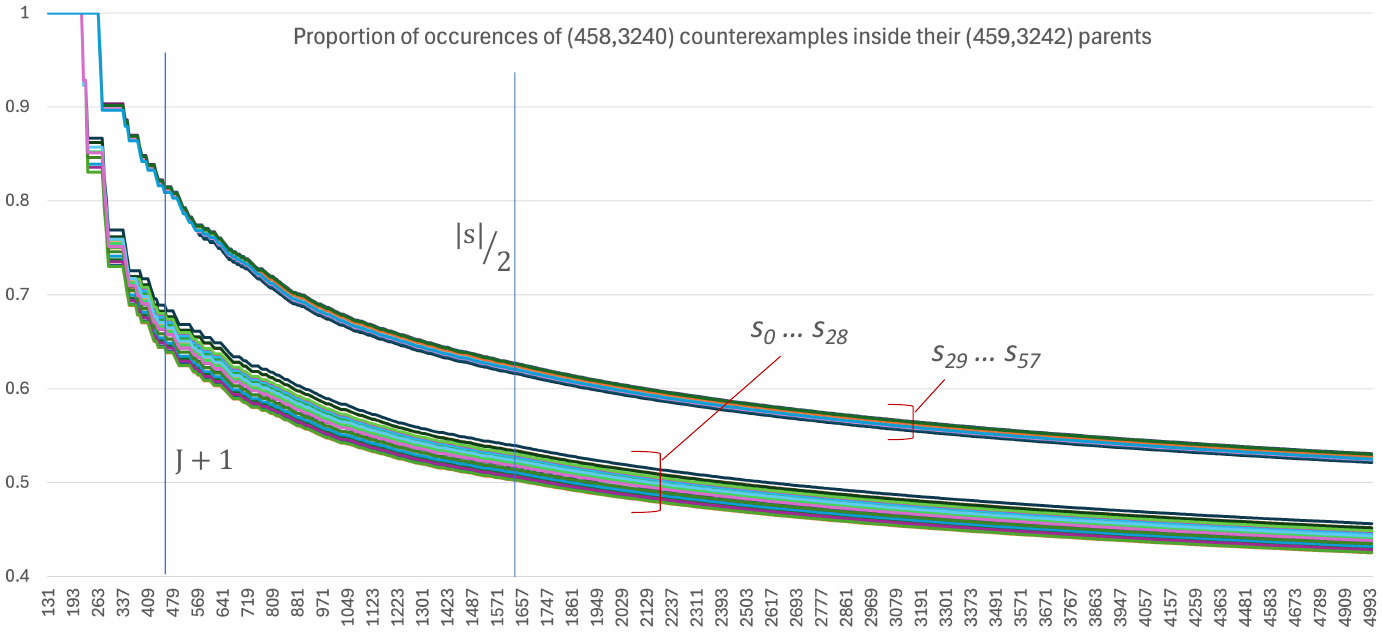}
\caption{\label{ParentsFig} Here we plot the ratio $n_{\rm in}(p^\#) / n_{\rm out}(p^\#)$ for each of the $(458,3240)$ constellations.
For the $116$ constellations there are $58$ distinct curves.  By symmetry of $\pgap(p^\#)$, under reversal of the constellations the curve for 
$s_j$ is the same as the curve for $s_{115-j}$.}
\end{figure}

\subsection{Asymptotic relative populations.}
From \cite{FBHktuple} the asymptotic relative population of a constellation $s$ is the product of two factors.
The first factor is the product of the number of admissible residues up through $k=J+1$, and the second factor runs from $J+1$ through half the span of $s$.
\begin{equation}\label{Eqwinf}
 w_{s,J}(\infty) =  \prod_{q \le J+1} (q-\nu_q) \, \cdot \, \prod_{J+1 < q \le |s|/2}\frac{q-\nu_q}{q-J-1} 
 \end{equation}

The asymptotic relative populations and these two factors are listed in Table~\ref{EngWinfTbl} for the Engelsma $(458,3240)$ constellations.

\begin{table}
\centering
\begin{tabular}{cccc || cccc}
$ic$ & $\prod_{q \le 459}$ & $\prod_{q > 459}$ & $w_{s}(\infty)$ &
$ic$ & $\prod_{q \le 459}$ & $\prod_{q > 459}$ & $w_{s}(\infty)$ \\ \hline
 0,	115	& 1.98{\em E}72 & 5.55{\em E}16 & 1.101{\em E}89	& 29,	86 & 1.57{\em E}72 & 5.65{\em E}16	& 8.896{\em E}88	\\
1,	114	&	6.15{\em E}72	&	5.00{\em E}16	&	3.075{\em E}89	&	30,	85	&	4.88{\em E}72	&	5.14{\em E}16	&	2.507{\em E}89	\\
2,	113	&	6.48{\em E}72	&	4.90{\em E}16	&	3.176{\em E}89	&	31,	84	&	5.18{\em E}72	&	5.01{\em E}16	&	2.597{\em E}89	\\
3,	112	&	3.97{\em E}72	&	5.02{\em E}16	&	1.993{\em E}89	&	32,	83	&	3.19{\em E}72	&	5.14{\em E}16	&	1.639{\em E}89	\\
4,	111	&	2.94{\em E}72	&	4.80{\em E}16	&	1.411{\em E}89	&	33,	82	&	2.36{\em E}72	&	4.91{\em E}16	&	1.157{\em E}89	\\
5,	110	&	2.27{\em E}72	&	5.03{\em E}16	&	1.144{\em E}89	&	34,	81	&	1.80{\em E}72	&	5.15{\em E}16	&	9.270{\em E}88	\\
6,	109	&	9.54{\em E}72	&	5.70{\em E}16	&	5.444{\em E}89	&	35,	80	&	8.07{\em E}72	&	5.80{\em E}16	&	4.685{\em E}89	\\
7,	108	&	5.19{\em E}72	&	5.73{\em E}16	&	2.974{\em E}89	&	36,	79	&	4.35{\em E}72	&	5.83{\em E}16	&	2.538{\em E}89	\\
8,	107	&	1.13{\em E}73	&	5.61{\em E}16	&	6.357{\em E}89	&	37,	78	&	9.55{\em E}72	&	5.70{\em E}16	&	5.445{\em E}89	\\
9,	106	&	5.94{\em E}72	&	5.66{\em E}16	&	3.358{\em E}89	&	38,	77	&	4.90{\em E}72	&	5.73{\em E}16	&	2.811{\em E}89	\\
10,	105	&	6.80{\em E}72	&	5.70{\em E}16	&	3.878{\em E}89	&	39,	76	&	5.62{\em E}72	&	5.78{\em E}16	&	3.251{\em E}89	\\
11,	104	&	1.36{\em E}73	&	5.58{\em E}16	&	7.575{\em E}89	&	40,	75	&	1.12{\em E}73	&	5.65{\em E}16	&	6.342{\em E}89	\\
12,	103	&	1.45{\em E}73	&	5.63{\em E}16	&	8.162{\em E}89	&	41,	74	&	1.20{\em E}73	&	5.70{\em E}16	&	6.866{\em E}89	\\
13,	102	&	1.50{\em E}73	&	5.65{\em E}16	&	8.487{\em E}89	&	42,	73	&	1.24{\em E}73	&	5.75{\em E}16	&	7.145{\em E}89	\\
14,	101	&	1.72{\em E}73	&	5.72{\em E}16	&	9.859{\em E}89	&	43,	72	&	1.43{\em E}73	&	5.83{\em E}16	&	8.310{\em E}89	\\
15,	100	&	1.92{\em E}73	&	5.53{\em E}16	&	1.061{\em E}90	&	44,	71	&	1.59{\em E}73	&	5.61{\em E}16	&	8.919{\em E}89	\\
16,	99	&	1.05{\em E}73	&	5.63{\em E}16	&	5.930{\em E}89	&	45,	70	&	8.78{\em E}72	&	5.70{\em E}16	&	5.010{\em E}89	\\
17,	98	&	1.20{\em E}73	&	5.68{\em E}16	&	6.836{\em E}89	&	46,	69	&	1.01{\em E}73	&	5.75{\em E}16	&	5.782{\em E}89	\\
18,	97	&	1.52{\em E}73	&	5.48{\em E}16	&	8.332{\em E}89	&	47,	68	&	1.23{\em E}73	&	5.57{\em E}16	&	6.871{\em E}89	\\
19,	96	&	1.82{\em E}73	&	5.75{\em E}16	&	1.044{\em E}90	&	48,	67	&	1.48{\em E}73	&	5.83{\em E}16	&	8.647{\em E}89	\\
20,	95	&	3.20{\em E}73	&	5.40{\em E}16	&	1.728{\em E}90	&	49,	66	&	2.60{\em E}73	&	5.49{\em E}16	&	1.426{\em E}90	\\
21,	94	&	3.90{\em E}73	&	5.65{\em E}16	&	2.205{\em E}90	&	50,	65	&	3.19{\em E}73	&	5.73{\em E}16	&	1.826{\em E}90	\\
22,	93	&	4.58{\em E}73	&	5.72{\em E}16	&	2.623{\em E}90	&	51,	64	&	3.74{\em E}73	&	5.84{\em E}16	&	2.184{\em E}90	\\
23,	92	&	3.87{\em E}73	&	5.45{\em E}16	&	2.112{\em E}90	&	52,	63	&	3.14{\em E}73	&	5.58{\em E}16	&	1.751{\em E}90	\\
24,	91	&	4.39{\em E}73	&	5.33{\em E}16	&	2.343{\em E}90	&	53,	62	&	3.58{\em E}73	&	5.43{\em E}16	&	1.944{\em E}90	\\
{\bf 25,	90}	&	5.20{\em E}73	&	5.59{\em E}16	&	2.911{\em E}90	&	54,	61	&	4.27{\em E}73	&	5.68{\em E}16	&	2.422{\em E}90	\\
26,	89	&	2.65{\em E}73	&	5.45{\em E}16	&	1.442{\em E}90	&	55,	60	&	2.17{\em E}73	&	5.54{\em E}16	&	1.203{\em E}90	\\
27,	88	&	3.14{\em E}73	&	5.66{\em E}16	&	1.774{\em E}90	&	56,	59	&	2.58{\em E}73	&	5.74{\em E}16	&	1.483{\em E}90	\\
28,	87	&	1.75{\em E}73	&	5.21{\em E}16	&	9.114{\em E}89	&	57,	58	&	1.43{\em E}73	&	5.30{\em E}16	&	7.605{\em E}89	\\ \hline
\end{tabular}
\caption{ \label{EngWinfTbl} The asymptotic relative populations are listed for all $116$ of the Engelsma $(458,3240)$ constellations.  The two factors,
for $p \le 459$ and for $459 < p \le 1620$ from Equation~\ref{Eqwinf}, are also listed.}
\end{table}

Note that $s_{25}$ and its reversal $s_{90}$ have asymptotic relative populations of $8.657E90$, $32.73$ times higher than
the asymptotic relative populations for $s_{29}$ and its reversal $s_{86}$.
Most of this variation occurs in the first factor, the number of instances for $q \le 459$.

Recall from our earlier work \cite{FBHPatterns, FBHktuple} that the relative populations can only be compared across admissible constellations 
of the same length $J$.

At face value these numbers $w_{s,J}(\infty)$ seem large, but we have very few meaningful examples of constellations of length $J=458$ 
with which to compare these values.  There is the repetition of the gap $g=457^\#$ of length $J=458$ and span ${|s|=1.00368 E190}$ (vs $3240$).  
This repetition has asymptotic 
relative population ${\phi(457^\#)\approx 1.99 E186}$.

% SECTION:  conclusion
\section{Conclusion}
Repeating our approach in \cite{FBHnonconvex}, we analyze the evolution of instances of the $116$ constellations identified by Engelsma \cite{Engtbl, Eng2} of length $J=458$ and span $|s|=3240$.  These are nonconvex constellations, by which we mean
$$ \pi(|s|) < J.$$
If any instance of any of these $(458,3240)$ constellations occurred among primes, this occurrence would provide a counterexample to the
convexity conjecture.

We are unable to provide any such instance among primes.  We do know \cite{FBHktuple} that as admissible constellations all of the
$(458,3240)$ constellations arise in the cycles of gaps $\pgap(p^\#)$ among the $p$-rough numbers.  
We track the evolution of instances of the constellations across the early stages of Eratosthenes sieve, 
and we calculate the values for their asymptotic relative populations. 

 No instance of any of the $116$ $(458,3240)$-counterexamples occur among primes before $9.7\,E73$.  At this point they are all still inside their
 $(459,3242)$ parents.

We note in particular the close relationship between the $116$ $(458,3240)$ constellations and the $58$ $(459,3242)$ constellations studied in
\cite{FBHnonconvex}.  Our calculations show that every instance of the $(458,3240)$ constellations occurs within a $(459,3242)$ constellation (its parent),
until $\pgap(227^\#)$ at least.  We study the slow shift from instances inside the parents to instances outside the parents.

% SECTION ==== BOUNDARY === scraps below ==== BOUNDARY ====
% == graph of Delta-Phi
%\begin{figure}[hbt]
%\centering
%\includegraphics[width=5in]{Eng458_25_vpi.pdf}
%\caption{\label{Eng25vpiFig} The $(458,3240)$-counterexample $s_{25}$ plotted (in blue) as a segment of 
%$\Delta \Phi(x,\mu)$, compared to the
%gaps between primes (in red) from $0$ to $x$. The small gaps in $\pi(x)$ create a huge early lead, but as larger gaps occur between primes,
%it gives narrow constellations an opportunity to overtake it.   $s_{25}$ first crosses above the red graph for $(458,3240)$, and stays above for $(459,3242)$.}
%\end{figure}

%We graph the constellation $s_{25}$ against $\pi(x)$ in Figure~\ref{Eng25vpiFig}. 
%relative population among the $(458,3240)$-counterexamples.

\bibliographystyle{alpha}
\bibliography{primes2024}

\end{document}